\documentclass[10pt,letterpaper,dvips]{article}
 \usepackage{amssymb,latexsym,amsmath,amsthm}
\usepackage{fullpage} 
\newtheorem{thm*}{Theorem}

\newtheorem{thm}{Theorem}[section]
\newtheorem{prop}{Proposition}[section]

\newtheorem{dfn}{Definition}[section]

\newtheorem{lemma}{Lemma}[section]

\newtheorem{remark}{Remark}[section]

\newtheorem{cor}{Corollary}[section]

\newtheorem{obs}{Observation}[section]

\def\IR{{\mathbb{R}}}

\newcommand{\E}[0]{ \varepsilon}

\newcommand{ \pOm}{\partial \Omega}

\newcommand{\lam}{ \lambda} 
\newcommand{\alp}{ \alpha} 

\begin{document}

\title{Estimates on Pull-in Distances in MEMS Models and other  Nonlinear Eigenvalue Problems}
\author{Craig COWAN\thanks{This work was partially supported by a U.B.C. Graduate Fellowship, and is part of the author's PhD dissertation in preparation under the supervision of N. Ghoussoub.}\ \ \ and \ Nassif GHOUSSOUB \thanks{Partially supported by the Natural Science
and Engineering Research Council of Canada. }\\
Department of Mathematics, University of British Columbia,\\
Vancouver, B.C. Canada V6T 1Z2
\\ }

\date{\today}

\smallbreak
\maketitle

\begin{abstract} 

\end{abstract} Motivated by certain mathematical models for  Micro-Electro-Mechanical Systems (MEMS),  we give upper and lower $L^\infty$ estimates for the minimal solutions of nonlinear eigenvalue problems of the form $-\Delta u = \lambda  f(x) F(u)$ on a smooth bounded domain $ \Omega$ in $\IR^N$. We are mainly interested in the {\it pull-in distance}, that is the $L^\infty-$norm of the extremal solution $u^*$ and how it depends on the geometry of the domain, the dimension of the space, and the so-called {\it permittivity profile} $f$. In particular, our results provide mathematical proofs for various observed phenomena, as well as rigorous derivations for several estimates obtained numerically by Pelesko \cite{P}, Guo-Pan-Ward \cite{GPW} and others in the case of the MEMS non-linearity $F(u)=\frac{1}{(1-u)^2}$ and for power-law permittivity profiles $f(x)=|x|^\alpha$.

\section{Introduction} 

We examine  problems of the form
$$
\arraycolsep=1.5pt
\left\{ \begin{array}{lll} \hfill -\Delta u &=& \lambda  f(x) F(u)  \quad  \hbox{in }\Omega ,\\ [2mm]
\hfill  u&=&0 \qquad \quad \quad \quad\, \hbox{on }\partial\Omega,
\end{array} \right. \eqno{(P_{\lambda, f})}
$$
where $ \Omega$ is a bounded domain in $ \IR^N$, $ 0 < \lambda$, $ f$ is a nonnegative nonzero bounded H\"older continuous function, usually dubbed as the \emph{permittivity profile},  and $ F$ is a smooth, increasing, convex nonlinearity on its domain $0\in D_F\subset \IR$, such that $ F(0)=1$ and which blows up at the endpoint of its domain.  We shall concentrate on the two cases where either $F$ is superlinear and its domain is $D_F:=[0, +\infty)$ in which case $F$ is said to be {\it  a regular nonlinearity}, or  when $D_F:=[0, 1)$ and $ \lim_{ u \nearrow 1} F(u) =+\infty$ in which case, we say 
that $F$ is a {\it singular non-linearity}. Typical regular nonlinearities are $F(u)=e^u$ or $F(u)=(1+u)^p$ for $p>1$, while singular nonlinearities include $F(u)=(1-u)^{-p}$ for $p>0$. 

We say that a solution $ u $ of $ (P_{\lambda,f})$ is \emph{classical} provided $ \| u\|_{L^\infty} < \infty$ (resp., $\|u\|_{L^\infty}<1$) if $ F$ is a regular (resp., singular) nonlinearity.    Note that by elliptic regularity theory, this is equivalent to saying that a classical solution is in $C^{2,\alpha}$ for some $\alpha>0$. 

We shall also need to consider {\it $H^1_0-$weak solutions} of $ (P_{\lambda,f})$ which are those $u$ in  $H^1_0(\Omega)$ such that 
\begin{equation} \label{focus1}
\hbox{$\int_\Omega \nabla u \nabla \phi\, dx=\int_\Omega \lambda fF(u) \phi \, dx$
\quad for all $\phi \in H_0^1(\Omega)$.}
\end{equation}

It is by now well-known that -- regardless whether $F$ is a regular or singular nonlinearity -- there exists an extremal parameter $\lambda^*\in (0, +\infty)$ depending on $\Omega$, $f$ and $N$, and which can be defined as 
\[ 
\lambda^*(\Omega, f):= \sup\{\lambda >0: (P_{\lambda, f}) \mbox{ has a classical solution} \},
\]  
such that $(P_{\lambda, f})$ has a \emph{minimal} classical solution $u_\lambda$ for every $\lambda \in (0, \lambda^*)$, and no weak solution for $\lambda >\lambda^*.$ By a {\it ``minimal solution" $u$}, we mean one such that  any other solution $ v$ of $(P_{\lambda, f})$ satisfies $ v \ge u$ a.e. in $ \Omega$. One can then also show that $ \lambda \mapsto u_\lambda(x)$ is increasing on $(0,\lambda^*)$ for each $ x \in \Omega$.  This allows us to define the \emph{extremal solution} by 
\[ 
u^*(x):=\lim_{ \lambda \nearrow \lambda^*} u_\lambda(x), 
\] 
 which then can be shown to be the unique (weak) solution of $ (P_{\lambda^*,f})$. 

We shall also need  the notion of {\it stability}. Given a weak solution $ u $ of $(P_{\lambda,f})$, we say that $ u $ is \emph{stable} (resp., \emph{semi-stable}) provided $ \mu_1( \lambda,u) >0$, (resp., $ \mu_1(\lambda,u) \ge 0$) where 
\[ 
\mu_1( \lambda,u):= \inf \left\{ \int_\Omega  (| \nabla \psi|^2 - \lambda f(x)F'(u) \psi^2) dx: \; \psi \in H_0^1(\Omega), \int_\Omega \psi^2=1 \right\}.
\]  
Under our assumptions on the nonlinearity $F$, and whether it is regular or singular, one can show that for all $ 0< \lambda < \lambda^*$ the minimal solution $ u_\lambda$ is stable,  and consequently  that $ u^*$ is semi-stable.  If in addition, $ u^*$ is a classical solution of $(P_{\lambda^*,f})$,  then necessarily $ \mu_1( \lambda^*,u^*) =0$,  since otherwise one could use the Implicit Function Theorem, in a suitable function space, to obtain solutions to $(P_{\lambda,f})$ for $ \lambda > \lambda^*$,  which would be a contradiction. 
On the other hand,  one has the following useful result,  which was proved by Brezis-Vasquez \cite{BV} for regular nonlinearities, and by Ghoussoub-Guo \cite{GG} in the case of singular nonlinearities with general permittivity profiles. 

\begin{prop} \label{unique}
 A semi-stable $H^1_0(\Omega)-$weak  solution of $(P_{\lambda,f})$ that is not a classical solution can only occur at $\lambda^*$, in which case it must be equal to the extremal solution $u^*$. 
 \end{prop}

 The question of the regularity of the extremal solution has attracted a lot of attention in the last decade.  For general regular nonlinearities the extremal solution is classical provided one of the following holds:

\begin{itemize} 

\item $ \Omega$ is contained in $ \IR^N$ with $ N \le 3$ (Nedev, see \cite{Nedev}). 

\item $ \Omega$ is a ball in $ \IR^N$ with $ N \le 9$ (Cabre and Capella, see \cite{CC}).

\end{itemize}  The second result is optimal after one considers $ F(u)=e^u$ on the unit ball in $ \IR^{10}$.   It is an open question as to whether for $ 4 \le N \le 9$,  there is a regular nonlinearity $F$ and a domain $\Omega\subset \IR^N$ on which the corresponding extremal solution is unbounded.    In the case of the MEMS model, where $ F(u)=(1-u)^{-2}$, it is known that the extremal solution is classical provided $ N \le 7$ and that this result is optimal (see \cite{GG}).  On the other hand, for any dimension $N>2$, there exists a singular nonlinearity, namely $F(u)=(1-u)^{-p}$ for some $p:=p(N)>0$,  such that  the corresponding extremal is not classical (see  Chapter 3 of \cite{EGG}).

In this paper, we are mostly interested in the quantitative aspects of the regularity of the extremal solution $u^*$, which were initially motivated by the   equation
$$
\arraycolsep=1.5pt
\left\{ \begin{array}{lll} \hfill -\Delta u &=&  \frac{ \lambda f(x)}{(1-u)^2}  \quad  \hbox{in }\Omega ,\\ [2mm]
\hfill  u&=&0 \quad \quad \quad\, \hbox{on }\partial\Omega.
\end{array} \right. \eqno{(M_{\lambda, f})}
$$
In dimension $ N=2$ this  equation models a simple \emph{Micro-Electromechanical-Systems} MEMS device, which roughly consists
of a dielectric  elastic membrane that is attached to the boundary of $ \Omega$, and whose upper surface has a thin conducting film.  At a distance of $1$ above the undeflected membrane sits a grounded plate, i.e., a plate held at zero voltage.  When a voltage $V>0$ is applied to the thin film of the membrane, it deflects towards the ground plate. After various physical limits of the parameters involved, a dimensional argument and a simplification, ones arrives at $(M_{\lambda,f})$ for the steady state of the membrane.   Here $ \lambda$ is proportional to the applied voltage $V$ and the \emph{permittivity profile} $ f(x)$ allows for varying dielectric properties of the membrane.  

As seen above,  one expects  the extremal solution $ u^*$ in small dimension $N$ to be bounded away from $1$, hence to be a classical solution.  Since the parameter $\lambda^*$ corresponds to the critical  voltage beyond which there is a snap-through, and since $u^*$ is the optimal deflection of the membrane, it is therefore important for 
 the design of MEMS devices to know how the critical voltage $ \lambda^*$ and the \emph{pull-in-distance} -- defined as $\| u^*\|_{L^\infty}$ -- depend on the geometry of the membrane and on the permittivity profile. Several analytical and numerical estimates on $\lambda^*$ have been derived by Pelesko \cite{P}, Guo-Pan-Ward \cite{GPW}, Guo-Ghoussoub \cite{GG} and others in the case of the MEMS non-linearity $F(u)=\frac{1}{(1-u)^2}$. On the other hand, only numerical estimates have been obtained for the pull-in distance  in the case of power-law (resp., exponential) permittivity profiles $f(x)=|x|^\alpha$ (resp., $f(x)=e^{\alpha x}$).    In this paper, we shall see that one can give rigorous proofs and estimates for  phenomena,  which so far have only been  observed numerically by various authors. We shall also include corresponding results for general  -- not necessarily MEMS-type -- nonlinearities. 
 
 Here is a brief description of the paper. In section 2, we give upper estimates on the pull-in voltage $\lambda^*(\Omega, f)$ in fairly general situations, which will in turn yield lower bounds on $ \|u^*\|_{L^\infty}$.  What is remarkable here is that the estimates -- which are valid for general nonlinearities -- turn out to only depend on the permittivity profiles and not on the domain, nor on the dimension. Actually, they also apply to any reasonable uniformly elliptic operator. 
 
 In section 3, we give upper estimates on $ \|u^*\|_{L^\infty}$ which are computationally friendly.  Just as in the proof of the regularity of $u^*$ in low dimensions,  we use the energy estimates on the minimal solutions coupled with $L^p$ to $L^\infty$ Sobolev-type constants related to corresponding linear equations.   While the result is satisfactory for exponential nonlinearity, it is not so for the MEMS model, which led us to  reconsider this nonlinearity in the case of the ball where more precise $L^p$ to H\"older estimates can be used. We stress here that we are not interested in optimal upper estimates but rather estimates which,  if given a specific domain $ \Omega$ and a nonlinearity $F$,  one can easily obtain some numerical parameters by plotting a function of a single variable -- possibly -- using a Computer Algebra System.

Section 4 was motivated by an intriguing phenomena observed numerically by Guo-Pan-Ward \cite{GPW}, namely that on a two dimensional disc, the pull-in distance does not depend on the power of the permittivity profile $f(x)=|x|^\alpha$. We prove that this is indeed the case by a simple scaling argument which
relates the problem $(P_{\lambda,|x|^\alpha})$ on the unit ball of  $\IR^N$ to $(P_{\lambda,1})$ (which for simplicity we denote by $(P_{\lambda})$) on a   ball in a fractional dimension $N(\alpha)$. (Note that when $ f$ is radial and $ \Omega$ is the unit ball in $ \IR^N$,   all stable solutions of $(P_{\lambda,f})$ are then radial and hence we can examine the problem in fractional dimensions).   One can then easily transfer many results established for $(P_\lambda)$ to $(P_{\lambda,|x|^\alpha})$.  This observation, combined for example with the results of Cabre and Cappella \cite{CC}, leads to new regularity results for the extremal solution associated with $(P_{\lambda,|x|^\alpha})$.
 
In section 5, we study the asymptotics  in $ \lambda$, and  we obtain upper and lower pointwise bounds on the minimal solutions $ u_\lambda$, in the case  where $ u^*$ is singular.  The upper estimates are valid on arbitrary domains and we restrict ourselves to radial domains for the lower estimates since more explicit bounds can then be found. For that, we show that $ \lambda \mapsto u_\lambda$ is actually convex, and we exploit the fact that  both $ u^*$ and $ \frac{d}{d \lambda} u_\lambda |_{ \lambda=\lambda^*}$  are  explicitly known  in the case where $\Omega$ is a ball and $u^*$ is singular.

We now list our main notation. For a nonlinearity $F$, we denote by $a_F$ the upper bound of the domain $D_F$, which means that 
$a_F:= \infty$ if $ F$ is regular, and $a_F:=1$ if $F$ is singular, in such a way that 
$D_F:=[0,a_F)$. 

We shall also associate to $F$ the numbers
\begin{equation} \label{para}
\hbox{$B_F:= \sup\limits_{\tau \in (0, a_F)} \frac{\tau}{F(\tau)}$ \quad and \quad $C_F:=\int_0^{a_F} \frac{d\tau}{F(\tau)}.$} 
\end{equation}
The ball of radius $ R$ centred at $ x_0$ in $ \IR^N$ will be denoted by 
$B_R(x_0)$.  If $ x_0=0$ then we omit $ x_0$ and if $ R=1$ then we just write $B$. 
Given a set $ \Omega$ in $ \IR^N$ we let $ |\Omega|$  denote its $N$-dimensional Lebesgue measure, while $ \omega_N$ denotes the volume of the unit ball $B$ in $ \IR^N$.  The conjugate index of $ p$
will be denoted by $p'$ in such a way that $ \frac{1}{p}+ \frac{1}{p'}=1$. 
For a radial function $ u$ we write $ u(r)=u(|x|)$.   The first eigenvalue of $ -\Delta $ in $H_0^1(\Omega)$
will be denoted by $ \lambda_1(\Omega)$ and the corresponding positive eigenfunction will be $ \phi_\Omega$, assuming the normalization $ \int_\Omega \phi_\Omega =1$. 

\section{Lower estimates for the $L^\infty-$norm of the extremal solution} 

This section is devoted to the proof of the following result.

 \begin{thm}\label{main.lower} Suppose $F$ is either a regular or singular nonlinearity and that $ u^*$ is the extremal solution  of $ (P_{\lambda,f})$, which we assume to be classical.  Then, 
 \begin{equation}
  \| u^*\|_{L^\infty} \geq (F')^{-1} \left( \max \left\{  
 \frac{1}{B_F}\frac{ \inf_\Omega f}{\sup_\Omega f}, \frac{1}{ C_F} \frac{ \int_\Omega f \phi_\Omega dx}{\sup_\Omega f} \right\} \right),
 \end{equation}
 where we define $ (F')^{-1}(z)=0$  for $ z < F'(0)$.

\end{thm}

Before proceeding with the proof, we give some applications.

\begin{cor} \label{3.stooges} Suppose $f$ is a non-negative bounded H\"older continuous permittivity profile and that the extremal solution $ u^*$ of $(P_{\lambda, f})$ on a bounded domain $\Omega$ is regular.   
\begin{enumerate}

 \item  If $ F(u) = \frac{1}{(1-u)^p}$, $p>0$, then 
 \begin{equation}
 \| u^*\|_{L^\infty} \geq  1-\min \left\{ \frac{p}{p+1}\Big(\frac{\sup_\Omega f}{\inf_\Omega f}\Big)^{\frac{1}{p+1}},  \Big(\frac{p}{p+1}\frac{\sup_\Omega f}{\int_\Omega f\phi_\Omega\, dx}\Big)^{\frac{1}{p+1}}   \right\}. 
\end{equation} 
In particular, when the permittivity $f\equiv 1$, then for any dimension $1\leq N\leq 2+\frac{4p}{p+1}+4\sqrt \frac{p}{p+1}$, and any bounded domain $\Omega \subset \IR^N$, we have 
\begin{equation}
 \| u^*\|_{L^\infty} \ge \frac{1}{p+1}.
 \end{equation} 
 \item If $ F(u)=(u+1)^p$, $ p>1$, then 
  \begin{equation}
 \| u^*\|_{L^\infty} \geq  \max \left\{   \frac{p}{p-1}\Big(\frac{ \inf_\Omega f}{\sup_\Omega f}\Big)^{\frac{1}{p-1}} , \Big(\frac{p-1}{p}\frac{ \int_\Omega f \phi_\Omega dx}{\sup_\Omega f}\Big)^{\frac{1}{p-1}}\right\}-1
 \end{equation}
  In particular, when $f\equiv 1$, then for any dimension $1\leq N\leq 10$, and any bounded domain $\Omega \subset \IR^N$, we have  $ \| u^*\|_{L^\infty} \ge \frac{1}{p-1}$. 
  \item If $ F(u) = e^u$, then 
  \begin{equation}
 \| u^*\|_{L^\infty} \geq  \max \left\{ 1+\log \Big(\frac{\inf_\Omega f}{\sup_\Omega f}\Big), \,  \log \Big( \frac{\int_\Omega f\phi_\Omega\, dx}{\sup_\Omega f}\Big) \right\}.
 \end{equation}
  In particular, when the permittivity $f\equiv 1$, then for any dimension $1\leq N\leq 9$, and any bounded domain $\Omega \subset \IR^N$, we have $ \| u^*\|_{L^\infty} \ge 1$. 
  \end{enumerate} 
\end{cor} 

 \begin{remark} \rm Note that the lower bounds (when $f=1$) are independent of the domain. It is also fairly easy to adapt the proof below to show that they are not particularly exclusive to the Laplacian $ -\Delta$.  Indeed, the same lower bounds can be obtained if we replace it by any  operator of the form $ L(u):=-{\rm div}(A(x) \nabla u)$ where $A(x)$ is a symmetric uniformly positive definite $N \times N$ matrix defined in $ \Omega$. 
 
 Moreover, the same arguments show that the extremal solution associated with 
 \begin{equation}
\hbox{$ \Delta^2 u = \lambda F(u)$ \qquad on $\Omega$},
 \end{equation}
 also satisfies the same lower bound, where for general domains $ \Omega$ we restrict our attention to the Navier boundary conditions: $ u=\Delta u =0$ on $ \pOm$, while in the case of $ \Omega$ being a ball we can use the Dirichlet boundary conditions: $ u = \partial_\nu u =0$ on $ \partial B$. For recent advances on fourth order nonlinear eigenvalue problems, we refer to \cite{CDG}, \cite{CEG}, and \cite{AGGM}.
 
 \end{remark} 
 The proof of Theorem \ref{main.lower} follows immediately from the combination of the following two propositions. The first provides  upper estimates on $ \lambda^*$, in terms of $F$, $\Omega$ and $f$.

\begin{prop} \label{voltage}Suppose $ F$ is either a regular or singular nonlinearity. Then 
\begin{equation}
\lambda^*(\Omega, f) \le  \lambda_1(\Omega)\min \left\{ \frac{B_F}{\inf_\Omega f},  \frac{C_F}{ \int_\Omega f \phi_\Omega dx} \right\}, 
\end{equation}
where $B_F$ and $C_F$ are given in $(\ref{para})$. 
\end{prop} 

\begin{proof} 
Supposing $ u $ is a classical solution of $ (P_{\lambda,f})$, we multiply both sides of the equation by $ \phi_\Omega$ and integrate to obtain 
 \[
  \int_\Omega \left( \lambda F(u) f - \lambda_1(\Omega) u \right) \phi_\Omega dx =0.
  \] 
 Since $ \phi_\Omega >0$ we must have 
 \[
  \lambda \le \lambda_1(\Omega) \sup_\Omega \frac{ u}{ f F(u)} \le \frac{ \lambda_1(\Omega)}{\inf_\Omega f} \sup_{z\in D_F} \frac{z}{F(z)}=\frac{ \lambda_1(\Omega)B_F}{\inf_\Omega f}.
 \]  
For the second bound, multiply $ (P_{\lambda,f})$ by $ \frac{ \phi_\Omega}{ F(u)}$   and integrate to obtain 
\begin{eqnarray*}
\int_\Omega \lambda f \phi_\Omega dx&=& \int_\Omega (-\Delta u) \frac{ \phi_\Omega}{F(u)} dx\\
&=& \int_\Omega \frac{ \nabla u \cdot \nabla \phi_\Omega}{F(u)} dx - \int_\Omega \frac{ \phi_\Omega F'(u) | \nabla u|^2}{ F(u)^2} dx\\
& \le & \int_\Omega\frac{ \nabla u \cdot \nabla \phi_\Omega}{F(u)} dx \\ 
&=& \int_\Omega \nabla \phi_\Omega \cdot \nabla \left( \int_0^{u(x)} \frac{1}{F(\tau)} d \tau \right) dx, \\
&=& \lambda_1(\Omega) \int_\Omega \phi_\Omega \left( \int_0^{u(x)} \frac{1}{F(\tau)} d \tau \right) dx \\
& \le & \lambda_1(\Omega) C_F,
\end{eqnarray*} after recalling the normalization of $\phi_\Omega$. 
\end{proof}

\begin{prop} Suppose $ u^*$ is the extremal solution of $ (P_{\lambda,f})$ which we assume to be  classical.  Then 
\begin{equation}
\lambda_1(\Omega) \le \lambda^* \| f F'(u^*)\|_{L^\infty}.
\end{equation}

\end{prop} 

\begin{proof} If $ u $ is a classical solution of $ (P_{\lambda,f})$ with $ \lambda_1(\Omega) \ge \lambda \| f F'(u)\|_{L^\infty}$, then the variational formulation of the first eigenvalue $ \lambda_1(\Omega)$ of the Laplacian yields
\[
\int_\Omega|\nabla \phi|^2 dx\geq  \lambda_1(\Omega)\int_\Omega \phi^2 dx \geq \lambda \| f F'(u)\|_{L^\infty}\int_\Omega \phi^2 dx\geq \lambda \int_\Omega f F'(u) \phi^2 dx, 
\]
which means  that $ u$ is then a stable solution of $ (P_{\lambda,f})$.  

  Assuming now that $ u^*$ is a classical solution then, as mentioned in the introduction, we necessarily have that $\mu_1(u^*)= 0$. Using the bifurcation theorem  of Crandall-Rabinowitz  \cite{CR},  one can then obtain a second branch of solutions $U_\lambda$ to $(P_{\lambda,f})$ for $ \lambda$ in a small interval $(\lambda^*-\E,\lambda^*)$.  Moreover these solutions are unstable in the sense that $ \mu_1( \lambda, U_\lambda)  <0$.   It then follows from the above that  $ \lambda_1(\Omega) \leq  \lambda \| f F'(U_\lambda)\|_{L^\infty}$.  Sending $ \lambda \nearrow \lambda^*$ gives the desired result. \end{proof}

\section{Upper estimates for the $L^\infty-$norm of the extremal solution} 

In this section we look for upper estimates on the extremal solution $ u^*$ associated with $(P_\lambda)$, where $F$ is one of the three linearities considered in Corollary \ref{3.stooges}, and where we take $ f(x)=1$ for simplicity.     The methods consist of combining the energy estimates -- which are critical in showing that the extremal solution is regular in low dimension -- with various $L^\infty$ and H\"older estimates for linear equations. 

The following simple observation can be useful when looking for upper estimates.  
\begin{obs} Suppose $ u^*$ is the extremal solution associated with $ (P_{\lambda})$ in $ \Omega$ with extremal parameter $ \lambda^*$.   Then the extremal solution associated with $(P_\lambda)$ in the domain $\Omega_\rho:= \rho \Omega$ (where $ \rho>0$) is given by $ v_\rho^*(x):= u^*( \frac{x}{\rho})$ with extremal parameter $ \lambda^*(\Omega_\rho)= \frac{\lambda^*(\Omega)}{\rho^2}$. 
 
 \end{obs} 
 
 \subsection{Upper estimates on general domains}   

We begin with the case of exponential nonlinearities.

\begin{thm} \label{regular} Suppose $ F(u)=e^u$, $ \Omega$ is a bounded domain in $ \IR^N$ and $ u^*$ is the extremal solution associated with $(P_\lambda)$. 

\begin{enumerate} 

\item If $ 3 \le N \le 9$, then 
\begin{equation}
 \| u^*\|_{L^\infty} \le \frac{ \lambda_1(\Omega) \beta_N}{e (N-2)}\Big(\frac{ | \Omega |}{\omega_N}\Big)^\frac{N}{2},
\end{equation}
 where 
\[ \beta_N:= \inf \left\{N^\frac{-1}{2t+1}  \Big(\frac{2t}{4t+2-N})^\frac{2t}{2t+1}\Big( \frac{4}{2-t}\Big)^\frac{1}{t}; \quad \frac{N-2}{4} <t<2 \right\}. \] 

\item If $ \Omega \subset B_\frac{1}{2} \subset \IR^2$, then 
\begin{equation}
 \| u^*\|_{L^\infty} \le  \frac{\lambda_1(\Omega)}{e} \inf \left\{ \Big( \frac{4}{2-t}\Big)^\frac{1}{t} \big(\frac{ |\Omega|}{2\pi}\big)^\frac{1}{2t+1} \Lambda \Big( \frac{2t+1}{2t} , (\frac{ |\Omega|}{\pi})^\frac{1}{2}\Big)^{^\frac{2t}{2t+1}}; \quad  0 <t<2 \right\}, 
 \end{equation}
  where we define for $ p>1$ and $ 0<R<1$, 
\[ \Lambda(p,R):= \int_0^R ( - \log(r))^p r dr.
\]
\end{enumerate} 
\end{thm}
\noindent Using a computer algebra system one can evaluate the constants and obtain: 
\[ 
\beta_3 = 1.9915, \quad \beta_4= 2.2324, \quad \beta_5 = 2.6689,  \quad  \beta_6 = 3.42269,   \]
\[  \beta_7= 4.81191, \quad \beta_8 = 7.9408166, \quad \beta_9= 19.0031.
\]  

\noindent Note that one can combine this upper estimate  with the previous lower estimate  on $ u^*$ to obtain the following lower bound on the first eigenvalue of the Laplacian on a bounded domain in $ \IR^N$ whenever $ 3 \le N \le 9$:   
\begin{equation}
\lambda_1(\Omega) \geq  \frac{e (N-2)}{\beta_N} \Big(\frac{\omega_N}{|\Omega|}\Big)^\frac{N}{2}.
 \end{equation}
We also  consider the case of a MEMS nonlinearity.

\begin{thm}  \label{non-radial}   Suppose $ F(u)=(1-u)^{-2}$, $ \Omega$ is a bounded domain in $ \IR^N$ and $ u^*$ is the extremal solution associated with $ (M_\lambda)$ in $ \Omega$.   If  $ 3 \le N \le 7$, then 
\begin{equation} \label{13}
\| u^*\|_{L^\infty} \le 1- e^{-\frac{\lambda_1(\Omega)\gamma_N}{2(N-2)}\big(\frac{ |\Omega|}{\omega_N}\big)^{\frac{2 }{N}}}, 
\end{equation}
where
\[
\gamma_N:=\inf \left\{ \frac{16N^\frac{-3}{2t+3}}{27}\left(\frac{2t}{4t+6-3N}\right)^{\frac{2t}{2t+3}} \left( \frac{4(2t+1)}{4t+2-t^2}\right)^\frac{2}{t};\quad \frac{3(N-2)}{4} <t<2+\sqrt{6} \right\}.
 \] 
\end{thm}

\begin{remark} Using a similar approach one can show that if $ u^*$ is the extremal solution associated with $(P_\lambda)$, in the case where $ F(u)=(u+1)^p$, $p>1$,  and $ N=3 $ or $N=4$ then 
\[ 
\| u^*\|_{L^\infty} \le \frac{ (p-1)^{p-1} \lambda_1(\Omega)\beta_{N,p} }{ p^p (N-2)} \Big(\frac{|\Omega|}{ \omega_N}\Big)^{\frac{2}{N}},
\] 
where 

 \[ \beta_{N,p} = \inf \left\{  \frac{ ( 2tp-p-t^2)^\frac{-p}{t} (2t-1)^{ \frac{2t-1}{2t+p-1} + \frac{p}{t}} (2p)^\frac{p}{t}}{ N^\frac{p}{2t+p-1} (4t+2p -2 -Np)^\frac{2t-1}{2t+p-1} } : \max\{ t_p^-, t_{N,p} \} <t <t_p^+  \right\},\] and where 
 \[ t_p^-:= p - \sqrt{p^2-p}, \quad t_p^+:=p+ \sqrt{p^2-p}, \quad  t_{N,p}:= \frac{pN}{4} - \frac{p}{2} + \frac{1}{2}.\]
 We have omitted $ N=2$ just for simplicity.  To obtain estimates for $ N \le 10$ one has to perform a bootstrap argument or restrict the range of values for $ p$.  
 \end{remark}

For proving the above theorems we shall  need the following easy lemmas. 

\begin{lemma} \label{integrals} Let $\Omega$ be a smooth bounded domain in $\IR^N$. 

\begin{enumerate} 

\item If $ N \ge 3$ and $ \tau > \frac{N}{2}$, then for all $x \in \Omega$,
\[
 \left( \int_\Omega \frac{1}{|y-x|^{(N-2) \tau'}} dy \right)^\frac{1}{\tau'} \le \frac{ \omega_N^{1-\frac{2}{N}} N^{ 1 - \frac{1}{\tau} } (\tau-1)^\frac{\tau-1}{\tau} | \Omega|^{ \frac{2}{N}- \frac{1}{\tau}}}{ (2 \tau-N)^\frac{\tau-1}{\tau}}.
 \]   

\item If $N=2$ and $ \Omega \subset B_{\frac{1}{2}} \subset \IR^2$, then for all $ x \in \Omega$, 
\[ 
\left( \int_\Omega ( - \log( |y-x|))^{\tau'} dy \right)^\frac{1}{\tau'} \le ( 2 \pi )^\frac{\tau-1}{\tau}  \Lambda \Big( \frac{\tau}{\tau-1}, \frac{ | \Omega|^\frac{1}{2}}{ \pi^\frac{1}{2}} \Big)^\frac{ \tau-1}{\tau}.
\] \end{enumerate}

\end{lemma}  
We now obtain $ L^\infty$ bounds on linear equations. 

\begin{lemma} \label{linear-estimates} Suppose $ -\Delta u = g(x) \ge 0$ in $ \Omega$ with $ u=0 $ on $ \pOm$ where $ \Omega$ a bounded domain in $ \IR^N$ and $ g$ is smooth.  

\begin{enumerate} 

\item If $ N \ge 3$, then for all $ \tau > \frac{N}{2}$,  
\[ \| u \|_{L^\infty} \le \frac{ \| g\|_{L^\tau} (\tau-1)^\frac{\tau-1}{\tau} |\Omega|^{ \frac{2}{N}-\frac{1}{\tau}}}{ N^\frac{1}{\tau} (N-2) \omega_N^\frac{2}{N} (2 \tau-N)^\frac{\tau-1}{\tau}}.\] 

\item If $ N=2$ and $ \Omega \subset B_\frac{1}{2}$, then for all $ \tau >1$,  
\[ \| u\|_{L^\infty} \le \frac{ \| g\|_{L^\tau} \Lambda \Big( \frac{\tau}{\tau-1}, \frac{ |\Omega|^\frac{1}{2}}{ \pi^\frac{1}{2}} \Big)^\frac{\tau-1}{\tau}}{ ( 2 \pi)^\frac{1}{\tau}}.
\] 
\end{enumerate}

\end{lemma} 

\begin{proof} In both cases, we let $ v(x)$ denote the Newtonian potential of $ g$, i.e., 
\[ 
v(x):= \frac{1}{ N(N-2) \omega_N} \int_\Omega \frac{ g(y)}{|y-x|^{N-2}} dy, 
\] 
for $ N \ge 3$ and 
\[ 
v(x):=\frac{1}{2 \pi} \int_\Omega ( -\log( |y-x|)) g(y) dy, 
\]  
for $ N=2$. Since $ 0 \le u(x) \le v(x)$ in $ \Omega$,  it suffices to show the desired $ L^\infty$ estimate on $ v$.   To do this, one uses (for $N\geq 3$) H\"older's inequality  to write
\begin{eqnarray*}
v(x) & \le & \frac{1}{ N(N-2) \omega_N} \| g \|_{L^\tau} \left( \int_\Omega \frac{1}{|y-x|^{(N-2)\tau '}} dy \right)^\frac{1}{\tau '}. 
\end{eqnarray*} 
and then use the integral estimate in the previous lemma.  
\end{proof}  

We now derive the energy estimates for stable solutions.  

\begin{lemma}\label{energy-estimates} Suppose $ u $ is a classical semi-stable solution of $(P_\lambda)$. 

\begin{enumerate} 

\item If $ F(u)=e^u$, then for all $ 0<t<2$, we have
\[ 
\| e^u\|_{L^{2t+1}} \le \Big(\frac{ 4}{2-t}\big)^\frac{1}{t} | \Omega|^\frac{1}{2t+1}.
\] 
\item If $ F(u)=(1-u)^{-2}$, then for all $ 0 <t< 2 + \sqrt{6}$, we have
\[
 \| (1-u)^{-2} \|_{L^{ t+ \frac{3}{2}}} \le  
\left( \frac{4(2t+1)}{4t+2-t^2}\right)^\frac{2}{t} |\Omega|^\frac{2}{2t+3}.
 \] 
\end{enumerate}  
\end{lemma} 

\begin{proof}  1)\,  Using the test function $ \psi:= e^{tu}-1$, where $ 0<t<2$, in the stability conditions gives
\[ \frac{ \lambda}{t^2} \int_\Omega e^u ( e^{tu}-1)^2 \le \int_\Omega e^{2tu} | \nabla u|^2.\]  Now testing $(P_\lambda)$ on $ \phi=e^{2tu}-1$ and rearranging,  gives
\[
 \int_\Omega e^{2tu} | \nabla u|^2 = \frac{ \lambda}{2t} \int_\Omega e^u ( e^{2tu}-1).
\] 
Comparing the last two inequalities and dropping some positive terms gives 
\[ 
\left( \frac{1}{t} - \frac{1}{2} \right) \int_\Omega e^{(2t+1)u} \le \frac{2}{t} \int_\Omega e^{(t+1)u},
\] and after an application of H\"older's inequality on the right one obtains
\begin{equation} \label{energy}
 \| e^u\|_{L^{2t+1}} \le \frac{ 4^\frac{1}{t}}{(2-t)^\frac{1}{t}} | \Omega |^\frac{1}{2t+1}.
 \end{equation}  
 
 2)\,  Take $ \psi:=(1-u)^{-2}-1$, $ \phi:=(1-u)^{-2t-1}-1$ and proceed as in 1) by putting $ \psi$ into the stability condition and testing $ (M_\lambda)$ on $ \phi$. We obtain 
  \[ \left( \frac{2}{t^2}-\frac{1}{2t+1} \right) \int_\Omega \frac{1}{(1-u)^{2t+3}} \le \frac{4}{t^2} \int_\Omega \frac{1}{(1-u)^{t+3}},\] after dropping a couple of positive terms.  H\"older's inequality then yields
  \begin{equation}\label{starting}
   \left( \frac{2}{t^2}-\frac{1}{2t+1} \right) \Big\| \frac{1}{1-u} \Big\|_{L^{2t+3}}^t \le \frac{4}{t^2} | \Omega |^\frac{t}{2t+3}.
   \end{equation}
 \end{proof}  

\noindent We now combine the energy estimates with the linear estimates to obtain upper estimates on $ u^*$.   \\ 

\noindent{\bf Proof of Theorem \ref{regular}:}   Use  Lemma \ref{linear-estimates} with $ g(x):= \lambda^* e^{u^*}$ and $ \tau=2t+1$ along with the  estimate $ \lambda^* \le \frac{ \lambda_1(\Omega)}{e}$ to arrive at an estimate of the form 
\[ \| u^*\|_{L^\infty} \le C(t,N,| \Omega|) \frac{\lambda_1(\Omega) }{e}\| e^{u^*}\|_{L^{2t+1}}, 
\]  
where $C(t,N,| \Omega|)$ is provided by Lemma \ref{linear-estimates}. Now replace the $L^p-$norm on the right using the energy estimates from Lemma \ref{energy-estimates} to arrive at the desired result.  The restrictions on $t$ are a result of the restrictions on $ \tau$ in the linear estimates along with the restrictions on $ t$ from the energy estimates. $ \hfill \Box$ \\

\noindent{\bf Proof of  Theorem \ref{non-radial}:} 
Let $ \Omega $ denote a bounded domain in $ \IR^N$ where $ 3 \le N \le 7$ and let $ u^*$ denote the extremal solution associated with $ (M_\lambda)$ in $ \Omega$. Since the reasoning works for any log-convex nonlinearity $F$ (i.e., $ u \mapsto \log(F(u))$ is convex), we  define $ v:= \log( F(u^*))$, and so
\[ 
-\Delta v = - \frac{d^2}{du^2} \log( F(u)) \big|_{u=u^*} | \nabla u^* |^2 + \lambda^* F'(u^*) \qquad {\rm in}\,\, \Omega, 
\] 
with $ v=0$ on $ \pOm$.   Since $F$ is log convex, the first term on the right is negative.  We now define $ w$ by 
\begin{eqnarray*}
-\Delta w &=& \lambda^* F'(u^*) \qquad \,\, {\rm in}\,\, \Omega, \\
w &=& 0 \qquad \qquad \qquad {\rm on}\,\, \pOm,
\end{eqnarray*} 
 and  so $ 0 \le v(x) \le w(x)$ a.e. in $ \Omega$ by the maximum principle.   Using the linear estimates from Lemma \ref{linear-estimates} with $ g(x):= \lambda^* F'(u^*)$  one has 
\[ 
 \|\log \frac{1}{(1-u)^2}\|_{L^\infty} =   \| v \|_{L^\infty}  \le \| w \|_{L^\infty}  \le \tilde{C}_\tau \lambda^* \| F'(u^*) \|_{L^\tau}=  \tilde{C}_\tau\lambda^* \Big \| \frac{1}{1-u^*} \Big\|_{L^{3\tau}}^3.
 \] 
Taking now $ \tau= \frac{2t}{3}+1>\frac{N}{2}$, we can then replace the $L^\tau-$norm on the right by using the energy estimates from Lemma \ref{energy-estimates}, which will give the desired conclusion.  
$ \hfill \Box$

\subsection{Upper estimates on radial domains} 

While the upper estimate on general domains obtained in the last subsection is quite satisfactory for the exponential nonlinearity, it is not so for the case of the MEMS nonlinearity. Indeed,  using Maple one sees that if $ \Omega:=(0,1)^3$ the unit cube in $ \IR^3$ (and so $ \lambda_1(\Omega) =3\pi^2$), Formula   (\ref{13}) would then give  that 
\begin{equation}
 \| u^*\|_{L^\infty} \le .993...,
 \end{equation}
  which is clearly not a very good upper estimate.  This is mainly due to the fact  that we drop a potentially large term in the proof of Theorem 3.2, when we replaced $v$ by $w$ in order to apply the linear estimate of Lemma 3.2.  Note that this was not needed for the exponential nonlinearity in the proof of Theorem 3.1. 

In this section we examine radial domains, where better results are available on $u^*$, at least in the case of $ F(u)=(1-u)^{-2}$.       One can also examine the exponential nonlinearity using this approach but we won't do this since the last section seems to give satisfactory results.   For simplicity, we shall also restrict our attention to the case of $f\equiv 1$. The main difference is that we use here H\"older estimates on linear equations versus the $L^\infty$ estimates of the last subsection.  

 For the remainder of this section we assume that $ \Omega$ is the unit ball $B$ in $ \IR^N$ and $ F(u)=(1-u)^{-2}$.  We define the following parameter:  

 \begin{equation*} 
\gamma(\tau,N):= \left\{
\begin{array}{lr}
\frac{\tau}{2\tau-1}& \qquad N=1 \\
 & \\
\frac{\tau}{4(\tau-1)}& \qquad N=2 \\
&\\
\frac{ (\tau-1)^\frac{\tau-1}{\tau}}{ (N-2) N^\frac{1}{\tau} (2 \tau-N)^\frac{\tau-1}{\tau}}& \qquad N \ge 3.
\end{array}
\right.
\end{equation*} 
\begin{lemma} Let $ u$ denote a smooth radially decreasing solution of $ -\Delta u = g(r) \ge 0$ in the unit ball $ B$ of  $ \IR^N$. If $\max \{ 1 , \frac{N}{2} \} < \tau < \infty$, then one has the estimate:  
\begin{equation}
 \hbox{$u(0) \ge u(R) \ge u(0) - \frac{ \gamma(\tau,N) \| g \|_{L^\tau}}{ \omega_N^\frac{1}{\tau}}R^{2-\frac{N}{\tau}}$\quad  for all $R\in (0,1)$.}
 \end{equation} 
\end{lemma} 

\begin{proof} When $N=1$,  we integrate the equation between $0$ and $r$, and apply H\"older's inequality to obtain $- u'(r) \le \frac{ \| g\|_{\tau} r^\frac{1}{\tau'}}{2}$. 
 Now integrate both terms  between $0$ and $R$, and use again H\"older's inequality to obtain the desired result. \\
 
 When $ N \ge 2$, we multiply the equation by $ r$ and integrate over $(0,R)$ to arrive at 
\[R( -u'(R)) +(N-2) (u(0)-u(R)) = \int_0^R r g(r) dr.
\]  
If now $N=2$, then one has 
\[
 R( - u'(R)) = \int_0^R r g(r) dr \le \frac{  \|g \|_{L^\tau}R^\frac{2}{\tau'}}{2 \pi^{1-\frac{1}{\tau'}}}.
 \]  Dividing by $R$ and integrating the result over $(0,R)$ gives the claim.  

Now take $ N \ge3$. Since $ -u'(R) \ge 0$ we can drop a term to arrive at 
\[
 (N-2) (u(0)-u(R)) \le   \int_0^R r g(r) dr 
 =  \frac{1}{N \omega_N} \int_{B_R} \frac{g(x)}{|x|^{N-2}} dx 
  \le  \frac{ \| g\|_{L^\tau}}{N \omega_N} \left( \int_{B_R} \frac{1}{|x|^{(N-2) \tau'}} dx \right)^\frac{1}{\tau'},
 \] 
 and then use Lemma \ref{integrals} to evaluate  the integral on the right and  finish the proof. 
\end{proof}  
We now come to the result which will yield our upper estimates on $ u^*$.

\begin{thm} \label{singular} Suppose $ u$ is a smooth semi-stable solution of $ (P_\lambda)$ on the unit ball $B$  in $ \IR^N$,  where $ 1 \le N \le 11$.  Then, for $ \max \left\{ 0 , \frac{N-3}{2} \right\} <t< 2 + \sqrt{6}$, we have 

\begin{equation} \label{stuff} 
\int_0^1 \frac{ R^{N-1}\, dR}{ \left( 1 -\| u\|_{L^\infty} + \frac{4 \lambda_1(B) \gamma( t + \frac{3}{2},N)}{27} \left( \frac{4(2t+1)}{4t+2-t^2}\right)^\frac{2}{t} R^\frac{4t+6-2N}{2t+3} \right)^{2t+3}}  \le \frac{1}{N}\left( \frac{4(2t+1)}{4t+2-t^2}\right)^\frac{2t+3}{t}.
\end{equation}
 \end{thm}  

\begin{remark} \rm Note that the above theorem only shows that $ \| u\|_{L^\infty} $ is bounded away from $1$ if $4t+6-2N \ge 2N-1$ which, once coupled with the other condition on $t$ cannot be satisfied in the higher dimensions. This is to be expected since the extremal solution $ u^* $ satisfies $ u^*(0)=1$ for $ N \ge 8$. 
\end{remark}  

\begin{proof} Suppose $ u $ is a smooth semi-stable (so radial) solution of $(P_\lambda)$.  Then,  the above linear estimate applied with $ g(r):= \lambda (1-u)^{-2}$, gives that for all $R\in (0,1)$, 
\[ 1-u(R) \le 1-u(0) + \frac{ \lambda \gamma(\tau,N) \| (1-u)^{-2}\|_{L^\tau} R^{2-\frac{N}{2}}}{\omega_N^\frac{1}{\tau}}.
\]  
Now use the upper bound $ \lambda^* \leq \frac{4 \lambda_1(\Omega)}{27}$ from Proposition \ref{voltage},  take $ \tau= t+ \frac{3}{2}$,  and replace the $L^\tau-$norm on the right via the energy estimate from Lemma \ref{energy-estimates}, to obtain
\[ 1-u(R) \le 1-u(0) + \frac{ 4 \lambda_1(B) \gamma(t+\frac{3}{2},N)}{27} \left( \frac{4(2t+1)}{4t+2-t^2} \right)^\frac{2}{t} R^\frac{4t+6-2N}{2t+3}.
\]  
This yields the inequality
\[ 
N \omega_N  \int_0^1 \frac{ R^{N-1}\, dR}{ \left( 1 -\| u\|_{L^\infty} + \frac{4 \lambda_1(B) \gamma( t + \frac{3}{2},N)}{27} \left( \frac{4(2t+1)}{4t+2-t^2}\right)^\frac{2}{t} R^\frac{4t+6-2N}{2t+3} \right)^{2t+3}} \le N \omega_N \int_0^1 \frac{R^{N-1}\, dR}{(1-u(R))^{2t+3}}.
\] 
But the right hand side is actually equal to $\| (1-u)^{-2} \|_{L^{t+\frac{3}{2}}}^{ t + \frac{3}{2}}$,  hence we can again use the energy estimate from Lemma \ref{energy-estimates} to majorize it and complete the proof.  

\end{proof} 

\begin{remark} Using Maple to approximate the integral in (\ref{stuff}) while optimizing over $ t$, we get the following estimates on the extremal solution $u^*$ of $(P_\lambda)$ on  the unit ball in $ \IR^N$. 
\begin{enumerate} 

\item If $N=1$, then $\| u^*\|_{L^\infty} \le .49 ...$ 

\item If $N=2$, then $ \| u^*\|_{L^\infty} \le .55...$
\end{enumerate}
\end{remark}

\noindent We now obtain some explicit upper bounds on $ u^*$ in dimensions $ N=1,2$.  For that, we define 
\[ 
C(t,N):= \frac{4 \lambda_1(B) \gamma( t + \frac{3}{2},N) }{27} \left( \frac{4(2t+1)}{4t+2-t^2} \right)^\frac{2}{t}.\]

\begin{cor} Suppose $ u^*$ is the extremal solution of $ (P_\lambda)$ on the unit ball in $\IR^N$.  

\begin{enumerate}

\item If $N=1$, then 
\[ \| u^*\|_{L^\infty} \le  1-\sup\left\{ \left(2 C(t,1) (t+1) \left( \frac{4(2t+1)}{4t+2-t^2} \right)^\frac{2t+3}{t}
+ \frac{1}{C(t,1)^{2+2t}} \right)^\frac{-1}{2t+2}: 0<t< 2+\sqrt{6} \right\}.\] 

\item If $ N=2$, then 
\[ \| u^*\|_{L^\infty} \le 1-\sup\left\{\left(C(t,2)^2 (t+1)\left( \frac{4(2t+1)}{4t+2-t^2} \right)^\frac{2t+3}{t}
+ \frac{2t+2}{C(t,2)^{2t+1}} \right)^\frac{-1}{2t+1}: \frac{1}{2} \le t < 2+\sqrt{6} \right\}.\]

\end{enumerate}

\end{cor}

 \begin{proof} 1)\,   For $ 0 < t < 2+\sqrt{6}$ one has $ \frac{4t+6-2N}{2t+3} \ge 1$ and so we can replace the power on $ R$ in (\ref{stuff}) by $ 1$, so as to be able to explicitly calculate the integral in (\ref{stuff}). One then drops a few positive terms to arrive at the desired result.   
 
 2)\,   For $ \frac{1}{2} \le t < 2 + \sqrt{6}$ one has $ \frac{4t+6-2N}{2t+3} \ge 1$, so again we replace the power on $R$ in (\ref{stuff}) by $1$ and carry on as in the first part. 
  \end{proof}

 \section{Effect of power-law profiles on pull-in distances}

Our goal in this section is to study  the effect of power-like permittivity profiles $f(x) = |x|^\alpha$ on the problem $(P_{\lambda, \alpha})$ (our notation for $(P_{\lambda, |x|^\alpha})$) on the unit ball $B=B_1(0)$.  Numerical results  -- in particular those obtained by Guo, Pan and Ward in \cite{GPW} for MEMS nonlinearities--  give lots of information, but the most intriguing one is their observation that on a 2-dimensional disc, the pull-in distance does not depend on $\alpha$, at least in the case where $ F(u)=(1-u)^{-2}$, and that the solution develops a boundary-layer structure near the boundary of the domain as $\alp$ is increased.  In other words,  the $L^\infty-$norm of the extremal solution of $(M_{\lambda, \alpha})$ is  independent of $ \alpha \ge 0$.  In this section, we shall give a simple proof  of this  observation and other interesting phenomena, which actually holds true for more general nonlinearities. 

 We first observe that since $r\to r^\alpha$ is increasing, the moving plane
method
of Gidas, Ni and Nirenberg \cite{GNN} does not guarantee the radial symmetry of all 
solutions
to $(S_{\lambda, f})$. However,  one can show as in \cite{GG} the following proposition.  

\begin{prop}\label{B:lemradial} Let $\Omega$  be a radially symmetric domain and assume $f$ is a radial profile on $\Omega$. Then, the minimal solutions of $(P_{\lambda, f})$ on $\Omega$ are
necessarily radially
symmetric and consequently
$$\lambda^*(\Omega, f)=\lambda^*_r(\Omega, f)=\sup \big\{\lambda;\,  (P_{\lambda, f})
\,\,
\hbox{\rm has a radial solution} \big\}.$$
Moreover, if $\Omega$ is a ball, then  any radial solution of $(P_{\lambda, f})$ attains
its maximum at $0$.
\end{prop}
\proof It is clear that $\lambda^*_r(\Omega, f)\leq
\lambda^*(\Omega, f)$, and the reverse will be proved if we
establish that every minimal solution of $(P_{\lambda, f})$ with $0<\lam
<\lam ^*(\Omega ,f)$ is radially symmetric. The recursive linear scheme  that is used to construct the minimal solutions, gives a radial function at each step, and the resulting limiting function is therefore  radially symmetric.

For a solution $u(r)$ on the ball of radius $R$, we have $u_r(0)=0$ and
$$
-u_{rr}-\frac{N-1}{r}u_r=\lam f (r) F(u)\quad \mbox{in} \quad
(0, R)\,.
$$
Hence, $-\frac{d(r^{N-1}u_r)}{dr}=\lam f(r)r^{N-1} F(u)\geq 0$, and
therefore $u_r<0$ in $(0,R)$ since $u_r(0)=0$. This shows that
$u(r)$ attains its maximum at $0$, and that -- just as in the case where $f\equiv 1$ --  we have $\|u^*\|_\infty =u^*(0)$. 
\endproof

It follows from this proposition that for radially symmetric domains $\Omega$ and
profiles $f$, the  extremal solution $u^*$ is necessarily radially symmetric and that the pull-in distance
  is nothing but $u^*(0)$. We shall denote by  $\lambda^*_\alpha(N)$ (resp., $u^*_\alpha$) the pull-in voltage (resp., the extremal solution) of $(P_{\lambda, f})$ when $f(x)=|x|^\alpha$,  and  $\Omega$ is the unit ball in $\IR^N$.

  We now make the following crucial observation.
  
 \begin{prop}\label{late} For any $\alpha >-2$, the change of variable $u(r)=w(r^{1+\frac{\alpha}{2}})$ gives a  correspondence between the radially symmetric solutions of the equation 
 \begin{equation} 
\left\{ \begin{array}{ll}\label{N}
-\Delta_N u =   \lambda (1+\frac{\alpha}{2})^2|x|^\alpha F(u)  & \hbox{in } B,\\
 \quad \quad \,\, u=0 & \hbox{on }\partial B, \end{array} \right.
\end{equation} 
in dimension $N$ and those of the equation 
 \begin{equation}  
\left\{ \begin{array}{ll}\label{Nalpha}
-\Delta_{\frac{2(N+\alpha)}{2+\alpha}} w =   \lambda F(w) & \hbox{in } B,\\
\,\, \qquad  \qquad w=0 & \hbox{on }\partial B, \end{array} \right.
\end{equation} 
in -- the potentially fractional -- dimension $N(\alpha)=\frac{2(N+\alpha)}{2+\alpha}$. Moreover, we have 
\begin{equation}\label{miracle}
\hbox{$\lambda^*_\alpha(N)=(1+\frac{\alpha}{2})^2\lambda^*_0(N(\alpha))$\quad  and \quad $u_\alpha^*(r)=w^*(r^{1+\frac{\alpha}{2}})$,}
\end{equation}
where $u_\alpha^*$ is the extremal solution for (\ref{N}) and $w^*$ is the extremal solution of (\ref{Nalpha}).
\end{prop}  
 \noindent {\bf Proof:}  Indeed, by noting that for a radially symmetric $u$, we have $\Delta_Nu=u''+\frac{N-1}{r}u'$,   a straightforward calculation gives that 
 \[
 \Delta_Nu(r)+(1+\frac{\alpha}{2})^2\lambda r^\alpha F(u(r))=(1+\frac{\alpha}{2})^2r^\alpha\Big(\Delta_{N(\alpha)}w(r^{1+\frac{\alpha}{2}})+\lambda F(w(r^{1+\frac{\alpha}{2}}\Big),
 \]
where   $N(\alpha)=\frac{2(N+\alpha)}{2+\alpha}$. The rest follows from the uniqueness of the extremal solutions. \endproof

The above transformation allows us to deduce many results for the case of a power-law profile, from  corresponding ones associated to constant profiles. The fact that it preserves the $L^\infty$-norm has consequences on the pull-in distance and on the role of the profile in the critical dimension. It does also give proofs for various intriguing phenomena displayed by the numerical results below, especially in the case of a two dimensional disc, where the transformation does not alter the dimension since then $N (\alpha)=2$.  

  The following corollary  summarizes these consequences.
  
  \begin{cor} \label{estim.pull.distance} With the above notations, the following hold:
 \begin{enumerate}
\item For any dimension $ N \ge 1$, we have 
 for $\alpha >>1$,  
\begin{equation}
\lambda^*_\alpha(N)\sim (1+\frac{\alpha}{2})^2\lambda_0^*(2).
\end{equation}

\item  If $N=2$, then  
\begin{equation}
\hbox{$\lambda^*_\alpha(2)= (1+\frac{\alpha}{2})^2\lambda_0^*(2)$\quad   and \quad  $\|u^*_\alpha\|_{L^\infty}=\|u^*_0\|_{L^\infty}$ for all $\alpha>-2$.} 
\end{equation}
\end{enumerate}
\end{cor}
\begin{proof}   1) From the  above proposition, we have
  $\lambda^*_\alpha(N)= (1+\frac{\alpha}{2})^2 \lambda^*_0(\frac{2N+2\alpha}{\alpha+2})$, and  
$
\lambda^*_0(\frac{2N+2\alpha}{\alpha+2})\sim \lambda^*_0(2)$ 
 whenever $\alpha$ is large.\\
2) follows from the fact that  for $N=2$, we then have $N_\alpha=2$ for each $\alpha$ which means that 
\begin{equation} 
\lambda^*_\alpha(2)=\frac{(\alpha+2)^2}{4} \lambda^*_0(2), 
\end{equation}
and the pull-in distance in dimension $2$ on the ball is $
\|u_\alpha^*\|_{L^\infty}=\|w^*\|_{L^\infty}$, 
where $u_\alpha^*(r)=w^*(r^{1+\frac{\alpha}{2}})$. The pull-in distance is therefore independent of $\alpha$. 
\end{proof}

 \begin{cor} \label{estim.pull.distance} The following estimates hold in a MEMS model with a power-law permittivity profile, i.e., if $F(u)=(1-u)^{-2}$ and $f(x)=|x|^\alpha$. 
 
\begin{enumerate}
\item For any dimension $ N \ge 1$, we have 
 for $\alpha >>1$,  
\begin{equation}
\lambda^*_\alpha(N)\sim 0.789(1+\frac{\alpha}{2})^2.
\end{equation}

\item  If $N=2$, then 
\begin{equation}
\hbox{$\lambda^*_\alpha(2)= 0.789(1+\frac{\alpha}{2})^2$\quad   and \quad  $\|u^*_\alpha \|_{L^\infty} =0.445$ for all $\alpha>-2$.} 
\end{equation}
\item If $1\leq N \leq 7$ or if $N\geq 8$ and $\alp > \alp_N:= \frac{3N-14-4\sqrt{6}}{ 4+2\sqrt{6} }$, then the extremal solution $u^*_\alpha$ of $(M_{\lambda, \alpha})$ on the ball is classical and the pull-in distance  $\|u^*_\alpha\|_{L^\infty} <1$.

\item If the dimension $ N \ge 8$, and $0 \le \alp \le \alp_N:= \frac{3N-14-4\sqrt{6}}{ 4+2\sqrt{6} }$, then  the extremal solution is exactly $u_\alpha^*(x)=1-|x|^{\frac{2+\alpha}{3}}$, which means that  
\begin{equation}\label{best}
\hbox{$\lam ^*_\alpha(N)=\frac{(2+\alp )(3N+\alp -4)}{9}$ \quad and \quad $\|u^*_\alpha\|_{L^\infty} =1$.}
\end{equation}

\end{enumerate}
\end{cor}
\begin{proof} 1) and 2) follow from the  above proposition and the fact that $\lambda_0^*(2)=0.789$ and $\|u^*_0\|_{L^\infty}=0.445$.

3) \, The extremal solution  $u^*_\alpha$ of $(M_{\lambda, \alpha})$ is regular if and only if  $ \|u^*_\alpha\|_{L^\infty}=\|w^*\|_{L^\infty}<1$, where $w^*$ is the extremal solution for $(M_\lambda)$ in dimension $N (\alpha)$. 
According to \cite{GG}, 
this happens if $\frac{N(\alpha)}{2} <
1+\frac{4}{3}+2\sqrt{\frac{2}{3}}$ which means that $\alp > \alp_N:= \frac{3N-14-4\sqrt{6}}{ 4+2\sqrt{6} }$.

4)\, Note first that $u^*(x)=1-|x|^{\frac{2+\alpha}{3}}$ is a $H_0^1(B)-$weak
solution of $(M_{\lambda, |x|^\alpha})$ for any $\alpha >4-3N$.
The voltage is then $\lam_\alpha(N)=\frac{(2+\alp )(3N+\alp -4)}{9}$. Since now $\|u^*\|_{L^\infty}=1$, then  by Proposition \ref{unique}, 
it remains only to
show that for all $\phi \in H^1_0(B)$, 
\begin{equation}\int_{B} |\nabla \phi |^2 \ge \int_{B} \frac{2\lam
|x|^\alpha}{(1-u^*)^3}
\phi ^2. \label{B:2:350}
\end{equation}
But  Hardy's inequality  gives for $N\geq 3$ that 
$
\int _{B} |\nabla \phi |^2 \geq  \frac{(N-2)^2}{4} \int_{B} \frac{ \phi
^2}{|x|^2}
$ 
for any $\phi \in H^1_0(B)$, which means that
(\ref{B:2:350}) holds whenever $2\lam_\alpha (N) \leq \frac{(N-2)^2}{4}$ or,
equivalently, if $ N \ge 8$ and $0 \le \alpha \leq \alp_N= \frac{
3N-14-4 \sqrt{6}}{ 4+2\sqrt{6} }.$
 \end{proof}
 
 The above scaling has also the following direct consequences.
 
   \begin{cor} Suppose $F$ is a regular nonlinearity and $ N < 10 + 4 \alpha$, then the extremal functional $ u_\alpha^*$ of $(P_{\lambda, \alpha})$ on the ball is classical.

  \end{cor} 
  
  \begin{proof} Cabre and Cappella \cite{CC} showed that  the extremal solution on the ball is always bounded for $ N \le 9$.  They were only interested in integer dimensions,  but an inspection of  their proof indicates that the same result holds for any fractional dimensions  $N <10$. Combining this with our observation in Proposition \ref{late} completes the proof.  To see that this is optimal one recalls that when $ F(v) =e^v$ the extremal solution is unbounded in $N=10$.  Using this fact and the change of variables above yields  the optimality of this result. 
   \end{proof}

   \begin{remark} \rm One can also use this change of variables to study permittivity profiles with negative powers (i.e., $f(x)=|x|^\alpha$ for $0> \alpha > -2$.  For example suppose $ F(u)=e^u$, then using the above change of variables, one can show that for a fixed $N$ ($3 \le N \le 9$),   the extremal solution associated with 
   \[
 \hbox{$  -\Delta u =|x|^\alpha e^u$ \quad on $B$,}
   \]
   is singular for $ \alpha \in (-2, \frac{10-N}{4}]$, while it is a classical solution for $ \alpha \in (\frac{10-N}{4},0)$.    
  \end{remark}

 \section{Asymptotic behavior of stable solutions near the pull-in voltage}
 
  We now establish pointwise upper and lower estimates on the minimal solutions $ u_\lambda$ in terms of $ \lambda,\lambda^*$, the extremal solution $ u^*$ and $\frac{d}{d \lambda} u_\lambda \big|_{\lambda=\lambda^*}$.  
        For simplicity we restrict our attention to $ F(u)=e^u$ and $ F(u)=(1-u)^{-2}$.   In addition we allow fractional dimensions for results on the unit ball since then,  one can apply the results of the previous section to deal with power-law profiles $ (P_{\lambda,\alpha})$.  We first recall that  by using Proposition \ref{unique}, one can show the following:
  \begin{itemize}
\item    If $ F(u)=e^u$, then $u^*(x)= \log( \frac{1}{|x|^2})$ is 
an extremal solution on  the unit ball in $ \IR^N$ at $\lambda^*=2N-4$,
  provided $N \ge 10$.   
  
  \item If $ F(u)=\frac{1}{(1-u)^2}$, then $u^*(x)=  u^*(x)= 1-|x|^\frac{2}{3}$ is 
an extremal solution on  the unit ball in $ \IR^N$ at $\lambda^*=\frac{6N-8}{9}$,
  provided $ N \ge \frac{14+\sqrt{6}}{3}$. 

  \end{itemize}

   \begin{thm} \label{asymptotes} Let $ u^*$ denote the extremal solution of $ (P_\lambda)$ on a smooth bounded domain $ \Omega$ in $ \IR^N$.  
   
   \begin{enumerate} 
   
   \item If $ F(u)=(1-u)^{-2}$, then for $ 0 < \lambda < \lambda^*$, we have 
 \begin{equation}
 u_\lambda(x) \le \left( \frac{\lambda}{\lambda^*} \right)^\frac{1}{3} u^*(x) \qquad \mbox{for a.e. $x \in \Omega$.}
  \end{equation}
   
Moreover, if  $ \Omega$ is the unit ball in $ \IR^N$ with $ N \ge  \frac{14+4 \sqrt{6}}{3} =7.93...$, then for  $ 0 < \lambda < \lambda^* = \frac{6N-8}{9}$ we have
 \begin{equation}
  1-|x|^\frac{2}{3} - \frac{3( \lambda^* - \lambda)}{(6N-8)} \left( |x|^{ \frac{-N}{2}+1  + \frac{ \sqrt{9N^2-84N+100}}{6}}-1 \right) \le u_\lambda(x) \le \left( \frac{\lambda}{\lambda^*} \right)^\frac{1}{3} (1-|x|^\frac{2}{3} ),
   \end{equation}
   for a.e. $ x \in \Omega$. 
   
   \item If $ F(u)=e^u$, then for $ 0 < \lambda < \lambda^*$, 
 \begin{equation}
   u_\lambda(x) \le \log \left( \frac{ \lambda^*}{\lambda^*-\lambda + \lambda e^{-u^*}} \right) \qquad \mbox{for a.e. $ x \in \Omega$.}
    \end{equation}
   
Moreover, if $ \Omega$ is the unit ball in $\IR^N$ with $ N \ge 10$, then for $ 0< \lambda < \lambda^* = 2N-4$ we have 
    \begin{equation}  \log( \frac{1}{|x|^2}) - \frac{ (\lambda^*- \lambda)}{(2N-4)} \left( |x|^{ \frac{-N}{2}+1+ \frac{ \sqrt{ N^2-12N+20}}{{2}   }   } -1 \right) \le  u_\lambda(x) \le \log \left(  \frac{ \lambda^*}{ \lambda^* - \lambda + \lambda |x|^2} \right), 
 \end{equation}
 for a.e. $ x \in \Omega$. 
   
   \end{enumerate}
   
   \end{thm} 
    \begin{proof}
     The upper estimates follow easily from the minimality of $u_\lambda$ and the fact that 
    $x \mapsto \left( \frac{ \lambda}{\lambda^*} \right)^\frac{1}{3} u^*(x)$ (resp., $x \mapsto  \log \left( \frac{ \lambda^*}{\lambda^*-\lambda + \lambda e^{-u^*}} \right)$ 
      is a supersolution of $ (P_\lambda)$ in the case that $ F(u)= (1-u)^{-2}$ (resp.,  $ F(u)=e^u$).

For the lower bound, we shall proceed as follows:
 First,  recall that $ \lambda \mapsto u_\lambda$ is differentiable and increasing on $ (0,\lambda^*)$, and so if one defines $ v_\lambda:= \frac{d}{d \lambda} u_\lambda$,  then $v_\lambda$ is positive and solves the linear equation
  \begin{equation*}  
\left\{ \begin{array}{ll}\label{derive}
    -\Delta v =   F(u_\lambda) + \lambda F'(u_\lambda) v  & \hbox{in } \Omega,\\
\,\,\, \quad v=0 & \hbox{on }\partial \Omega,   \end{array} \right. \hskip 50pt (Q_\lambda)
\end{equation*} 
where, $F$ is given by either $ e^u$ or $ (1-u)^{-2}$.   We shall need the following notion.
 \begin{dfn} An extremal solution $u^*$ associated with $ (P_\lambda)$ is said to be \emph{super-stable} provided there exists  $ \E>0$ such that 
 \[ 
 \hbox{$( \lambda^* +\E) \int_\Omega F'(u^*) \psi^2 \le \int_\Omega | \nabla \psi|^2$ \qquad for all $\psi \in H_0^1(\Omega).$}
 \]
 \end{dfn}
Note that if $ u^*$ is a super-stable extremal solution then $ \mu_1(\lambda^*, u^*) >0$. We shall see at the end of this section that the converse is however not true.  We first establish the following result.   
    
        \begin{lemma} Assume $\Omega$ is a smooth bounded domain in $\IR^N$. Then, 
    \begin{enumerate} 
    
   \item For $ 0< \lambda < \lambda^*$,  $v_\lambda$ is the unique $H_0^1-$weak solution of $(Q_\lambda)$.
   
   \item  $ \lambda \mapsto v_\lambda$ is increasing on $(0,\lambda^*)$, and therefore $v^*(x):= \lim_{ \lambda \rightarrow \lambda^*} v_\lambda(x)$ is defined for a.e. $x\in \Omega$.

   \item  $ \lambda \mapsto u_\lambda$ is  convex on $(0,\lambda^*)$, and therefore for $ 0 < \lambda < \lambda^*$ we have for  a.e. $x\in \Omega$,
      \begin{equation} \label{convex} u_\lambda(x) \ge u^*(x) + ( \lambda-\lambda^*) v^*(x).
    \end{equation}  
 
       \item If $ u^*$ is super-stable, then $ v^*$ is the unique $H_0^1-$weak solution of $(Q)_{\lambda^*}$.

    \end{enumerate}
      \end{lemma}  
       \begin{proof}  (1)\, One can use the fact that $ \mu_1( \lambda,u_\lambda)\geq 0$, and a standard minimization argument to show the  existence of an $H_0^1-$solution to $(Q_\lambda)$.  Using the fact that $ \mu_1( \lambda,u_\lambda)>0$ one can see that the solution is unique.

(2)\, Let $ 0< \lambda < \lambda^*$ and $ \E>0$ small. Note first that 
    \begin{eqnarray*}
    -\Delta ( v_{\lambda+\E} - v_\lambda) &=& F(u_{\lambda+\E}) - F(u_\lambda) +\E F'(u_{\lambda+\E}) v_{ \lambda+\E} \\
    &&+ \lambda F'(u_{\lambda+\E}) v_{ \lambda+\E} - \lambda F'(u_\lambda) v_\lambda \\
    &=& g(x) + \lambda F'(u_\lambda) (v_{\lambda+\E} -v_\lambda),
    \end{eqnarray*} where 
    \[
    g(x):=F(u_{\lambda+\E}) - F(u_\lambda) +\E F'(u_{\lambda+\E}) v_{ \lambda+\E}+ \lambda \big(F'(u_{\lambda+\E}) v_{ \lambda+\E} - F'(u_\lambda) v_{\lambda+\E}\big)
    \]
   is in $H^1(\Omega)$ and is positive.   Now set $ w:=v_{\lambda+\E} -v_\lambda$ in such a way that $ w$ solves 
    \begin{eqnarray*}
    -\Delta w &=& g(x) + \lambda F'(u_\lambda) w \qquad \hbox{on $\Omega$},  \\
    w&=& 0 \qquad \qquad  \qquad \qquad \quad \hbox{on $\pOm$}.
    \end{eqnarray*}   Testing this equation  on $ w^-$ gives 
    \[ - \int_\Omega g w^- \ge \mu_1(\lambda, u_\lambda) \int_\Omega (w^-)^2,\] and hence  $ w^-=0$ a.e. in $ \Omega$.   By the maximum principle one then get that $ w>0$ in $ \Omega$ and hence that $\lambda \to v_\lambda$ is increasing.    We can therefore define the limit $v^*(x):= \lim_{ \lambda \rightarrow \lambda^*} v_\lambda(x)$,  which exists a.e. $x$ in $ \Omega$, though it might be infinite on a large set.  
    
(3) The convexity of $ \lambda \mapsto  u_\lambda$ follows from the fact that $ \lambda \mapsto v_\lambda$ is increasing.  We can therefore write $ u_\lambda \ge u_t + (\lambda -t) v_t$ for $ 0 < \lambda,t < \lambda^*$ and a.e. $ x  \in \Omega$.  The claim now follows by letting  $t$ go to $\lambda^*$.

(4) Since $ u^*$ is super-stable one has 
   \[ 
   ( \lambda +\E) \int_\Omega F'(u_\lambda) \psi^2 \le \int_\Omega | \nabla \psi|^2 \qquad \forall \psi \in H_0^1.\] Using this and testing $ (Q_\lambda)$ on $ v_\lambda$ gives 
   \[ \E \int_\Omega F'(u_\lambda) v_\lambda^2 \le \int_\Omega F(u_\lambda) v_\lambda.
   \] 
Since $F$ is either $ F(u) = e^u$ or $ F(u)=(1-u)^{-2}$, the left hand side is necessarily bounded.    From this and again by testing $(Q_\lambda)$ on $v_\lambda$ one sees that $ v_\lambda$ is bounded in $H_0^1$.  Passing to limits, one sees that $ v^* $ is a $H_0^1-$weak solution of $(Q_{\lambda^*})$.  The uniqueness follows from the fact that $ \mu_1( \lambda^*, u^*)>0$. 
      \end{proof}
    
We now complete the proof of Theorem \ref{asymptotes}. For that we assume that $\Omega$ is the unit ball in $\IR^N$.   It is then easy to show using Hardy's inequality that the explicit extremal solutions for $ (P_{\lambda})$ given above, are super-stable provided $ N > 10$ (resp., $ N > \frac{14+4 \sqrt{6}}{3} =7.93...$) when $ F(u) = e^u$ (resp., $F(u)=(1-u)^{-2}$).  An easy calculation also shows that 
 \[v^*(x)= \frac{1}{2N-4} \left( |x|^{ \frac{-N}{2}+1+ \frac{ \sqrt{ N^2-12N+20}}{{2}   }   } -1 \right),
 \] 
(when $F(u)=e^u$) resp., 
 \[ 
v^*(x)= \frac{3}{6N-8} \left( |x|^{ \frac{-N}{2}+1+ \frac{  \sqrt{9N^2-84N+100}}{6}} -1 \right),
 \] 
( when $F(u)=(1-u)^{-2}$) 
 are $H_0^1-$weak solutions of $(Q)_{\lambda^*}$ in the respective cases, assuming the dimension restrictions above.  Using this and the earlier convexity result gives the desired lower bounds for $ N >10$ ($N > \frac{14+4 \sqrt{6}}{3}$) in the exponential and MEMS cases respectively.  To obtain the result for the critical dimensions one passes to the limit in $N$.  We omit the details. 
 \end{proof}


\begin{thebibliography}{99}  


\bibitem{BV} H. Brezis and L. Vazquez, \emph{Blow-up solutions of some
nonlinear elliptic problems}, Rev. Mat. Univ. Complut. Madrid 10 (1997),
no. 2, 443--469.




\bibitem{CC} X. Cabre and A. Capella, \emph{Regularity of radial minimizers
and extremal solutions of semilinear elliptic equations},  J. Funct. Anal.
238  (2006),  no. 2, 709--733.


\bibitem{CDG} D. Cassani, J. do O and N.  Ghoussoub, \emph{On a fourth
order elliptic problem with a singular nonlinearity}, Advances Nonlinear Studies, {\bf 9}, (2009), 177-197 

\bibitem{CEG}  C. Cowan, P. Esposito, N. Ghoussoub, \emph{ The critical
dimension for a fourth order elliptic problem with singular nonlinearity},
preprint (2008) 15 pp.


\bibitem{CR} M.G. Crandall and P.H. Rabinowitz, \emph{Some continuation and
variation methods for positive solutions of nonlinear elliptic eigenvalue
problems}, Arch. Rat. Mech. Anal., 58 (1975), pp.207-218.



\bibitem{AGGM}  J. Davila, L. Dupaigne, I. Guerra and M. Montenegro,   
\emph{Stable solutions for the bilaplacian with exponential nonlinearity},
SIAM J. Math. Anal. 39 (2007), 565-592.


\bibitem{EGG} P. Esposito, N. Ghoussoub and Y. Guo, \emph{Compactness along
the branch of semi-stable and unstable solutions for an elliptic problem
with a singular nonlinearity},  Comm. Pure Appl. Math. {\bf 60}  (2007), 1731--1768.

\bibitem{EGG.book} P. Esposito, N. Ghoussoub,  Y. J. Guo: {\it  Mathematical Analysis of Partial Differential Equations Modeling  Electrostatic MEMS}, Research Monograph, 
In press  (2009) 260 p.  


\bibitem{Decomp} F. Gazzola and H.-Ch. Grunau, \emph{Critical dimensions
and higher order Sobolev inequalities with remainder terms}, NoDEA 8
(2001), 35-44.


\bibitem{GG} N. Ghoussoub and Y. Guo, \emph{On the partial differential
equations of electro MEMS devices: stationary case}, SIAM J. Math. Anal. 38
(2007), 1423-1449.



\bibitem{GM} N. Ghoussoub and A.  Moradifam, 
\emph{Bessel Pairs and Optimal Hardy and Hardy-Rellich Inequalities},
(preprint) 2008

\bibitem {GNN} B. Gidas, W.~M. Ni and L.
Nirenberg, {\em Symmetry and related properties via the maximum
principle}, Comm. Math. Phys. {\bf 68} (1979), no. 3, 209--243.

\bibitem{GPW} Y. Guo, Z. Pan and M.J. Ward, \emph{Touchdown and pull-in
voltage behavior of a mems device with varying dielectric properties}, SIAM
J. Appl. Math 66 (2005), 309-338.



\bibitem{LY} F.H. Lin and Y.S. Yang, \emph{Nonlinear non-local elliptic
equation modelling electrostatic acutation}, Proc. R. Soc. London, Ser. A
463 (2007), 1323-1337.



\bibitem{MP} F. Mignot and J-P. Puel, \emph{ Sur une classe de problemes
non lineaires avec non linearite positive, croissante, convexe}, Comm.
Partial Differential Equations 5 (1980), 791-836.


\bibitem{Nedev} G., Nedev, \emph{Regularity of the extremal
solution of semilinear elliptic equations,}  C. R. Acad. Sci. Paris SŽr. I
Math.  330  (2000),  no. 11, 997--1002.



\bibitem{P} J.A. Pelesko, \emph{Mathematical modeling of electrostatic mems
with tailored dielectric properties}, SIAM J. Appl. Math. 62 (2002),
888-908.



\bibitem{BP}  J.A. Pelesko and A.A. Bernstein, \emph{Modeling MEMS and
NEMS}, Chapman Hall and CRC Press, 2002.


\bibitem{VZ} L.  Vazquez and E. Zuazua, \emph{The Hardy Inequality and the
Asymptotic Behaviour of
the Heat Equation with an Inverse-Square Potential}, J. Funct.
Anal.  173, 103-153 (2000).


\end{thebibliography}
\end{document}